\documentclass[a4paper,12pt]{amsart}

\usepackage[centertags]{amsmath}
\usepackage{amsfonts}
\usepackage{amssymb}
\usepackage{amsmath}
\usepackage{amsthm}
\usepackage{times}
\usepackage{graphicx}
\usepackage[mathcal]{euscript}
\usepackage[all]{xy}
\usepackage{centernot}

\DeclareGraphicsExtensions{.pdf,.png,.jpg}

\numberwithin{equation}{section}%%

\newcommand{\Z}{\mathbb Z}
\renewcommand{\Re}{\mathcal{R}_\varphi(\varepsilon)}

{\bf}{\it}
\newtheorem{theorem}{Theorem}[section]%
\newtheorem{lemma}[theorem]{Lemma}
\newtheorem{corollary}[theorem]{Corollary}
\newtheorem{proposition}[theorem]{Proposition}

\newcommand{\be}{\begin{equation}}
\newcommand{\ee}{\end{equation}}

\newtheorem*{mtthm}{Main Technical Theorem}
\usepackage{varioref}

\author[Marks Ruziboev]{Marks Ruziboev }
\address{International School for Advanced Studies (SISSA)\\
Via Bonomea 265, Trieste, Italy
\\International Center for Theoretical
Physics (ICTP)\\
Strada Costiera 11, Trieste, Italy} \email[Marks
Ruziboev]{mruziboe@sissa.it}

\title[Decay of correlations non-H\"older observables]
{Decay of correlations for invertible maps with
non-H\"older observables}

\date{November 27, 2014}



\begin{document}

\begin{abstract}
An invertible dynamical system with some hyperbolic structure is
considered. Upper estimates for the correlations of continuous
observables is given in terms of  modulus of continuity. The result
is applied to  certain H\'enon maps and Solenoid maps with
intermittency.
\end{abstract}

\maketitle

\section{Introduction}
Let $f:M\to M$ be a map preserving a probability measure $\mu.$ The
system $(f, \mu)$ is called mixing if
$$ |\mu(f^{-n}A\cap
B)-\mu(A)\mu(B)|\to 0\,\,\text{as}\,\, n\to \infty $$ for any
$\mu$-measurable $A$ and $B.$ One of the manifestations of chaos in
a dynamical systems is this mixing property.  In particular, speed
of the convergence above is a measure of the strength of chaos in
the system $(f, \mu).$ Known counterexamples show that in general
there is no specific rate at which this convergence to zero happens.
One way to overcome this difficulty is to generalize the problem by
defining the correlation of observables $\varphi,\psi:M\to \mathbb
R$
$$C_n(\varphi, \psi; \mu):=
\left|\int(\varphi\circ f^n)\psi d\mu-\int\varphi d\mu\int\psi
d\mu\right|.$$ In this way sometimes it is possible to obtain
specific rates of decay of correlations, for observables with
certain regularity.

The rates of decay of correlations have been studied extensively
over thirty years. In their pioneering works Sinai \cite{S1}, Ruelle
\cite{R}, Bowen \cite{B} showed that every uniformly hyperbolic
system admits a special invariant measure (that is now called
SRB-measure) with exponential decay of correlations for H\"older
continuous observables. It turns out that generalizing the above
results to the systems with singularities or to the systems with
weaker hyperbolic properties is a very challenging problem and there
much progress has been made in recent years by  many authors here we
give just a few of them \cite{AlvPin, BY2, BLS, Chern, ChM, DHL,
Dol, Gou, LY, Li1, Li2, Y1, Y2}.

Notice that in the above works observables are assumed to be
H\"older continuous, or functions of bounded variation. Hence a
natural question to ask is how much we can generalize the classes of
observables that still admit some decay rate?  Several papers
address this question  in the context of one-sided subshifts of
finite type on a finite alphabet for example, see \cite{BFG, FL,
KMS, I, P}. For a comprehensive discussion of shift maps and their
ergodic properties we refer to \cite{Ba}. Moreover, there are
results on non-invertible Young Towers,  for example \cite{BM, KMS,
Ly, M}, and results that apply directly to certain non-uniformly
expanding systems \cite{PY}. We emphasize that all of the results we
above are for non-invertible maps and in the invertible case the
only reference we found is \cite{Z} which gives interesting
estimates for billiards with non-H\"older observables.

The aim of this work is to fill this gap, and show that the decay
rate of correlations for continuous observables is given in terms of
natural quantities, such as the modulus of continuity,  for systems
that admit Gibbs-Markov-Young (GMY)-structure, to be described
below.

%We follow the strategy used in \cite{AlvPin, Y1}, and reduce the
%system into non-invertible system. The key fact is we are still able
%to approximate decay of correlations for continuous observables on
%the manifold through H\"older observables on the corresponding
%tower.

%\subsection{Related work}
The rest of the paper is organized as follows. In the next section
we give the definition of GMY-structure, state the main technical
theorem and its applications to H\'enon maps and Solenoid maps with
intermittent fixed point. In section 3 we give the proof of the main
technical theorem, and in the last section we prove the results for
the H\'enon maps and Solenoid maps.

\section{Statement of results}

Let $M$ be  a metric space and $\mathcal C(M)$ be the space of continuous
functions defined on it.  Define modulus of continuity for
$\varphi\in\mathcal  C(M)$ as
$$\mathcal{R}_\varphi(\varepsilon)=\sup\{|\varphi(x)-\varphi(y)|:\,\,
d(x, y)<\varepsilon\}.
$$
Obviously, $\Re\to 0 $ as $\varepsilon\to 0$ if and only if
$\varphi$ is uniformly continuous. If $\varphi$ is H\"older
continuous  with   H\"older exponent $\alpha$ then  $\Re\lesssim
\varepsilon^\alpha.$  \footnote{Throughout the paper we use the
notation $a\lesssim b$ if there exists a constant $C$ independent of
$\varepsilon$ such that $a\le Cb.$}  In general, $\Re$ can converge
to $0$ very slowly, and the rate of convergence is slower than the
rate of convergence of $\varepsilon$ to $0$ even for H\"older
observables. We refer to \cite{Hol} for  the examples of various
slow rates of convergence.

 Let $a_n$ be a sequence of
positive numbers such that $\lim_{n \to\infty}a_n=0,$ and let
$\mathcal F(M)$ be a Banach space of observables defined on $M.$ We
say $(f, \mu)$ has decay of correlations at rate $a_n$ for functions
in $\mathcal F(M)$ if for any $\varphi, \psi\in \mathcal F(M)$ there
exists a constant $C=C(\varphi, \psi, f)$ such that
$$C_n(\varphi, \psi; \mu)\le Ca_n.$$

Below we give  applications of our main technical result to H\'enon
maps and Solenoid maps with intermittency.

\subsection{H\'enon maps}
Let $T_{a, b}:\mathbb{R}\to \mathbb{R}$ denote  the H\'enon map i.e.
$$T_{a, b}(x, y)=(1-ax^2+y, bx).$$

In \cite{BC, BY1, BY2} it was shown that there exists positive
measure set  $\mathcal A$ of parameters $(a, b)$ such that for any
$(a, b)\in\mathcal A$ corresponding $T_{a, b}$ admits unique
SRB-measure which is mixing and the speed of mixing is exponential
for H\"older continuous observables.  Here we assume that $(a,
b)\in\mathcal A,$ and investigate the problem of decay of
correlations for continuous observables.  More precisely, let
$\mathcal C(\mathbb R^2)$ be the space of all continuous real valued
functions. Let
$$\mathcal R_{\varphi, \psi}(\varepsilon)=\max\{\mathcal
R_\varphi(\varepsilon), \mathcal R_\psi(\varepsilon)\}.$$

\begin{theorem}\label{main} Let $f=T_{a,b}$ and $(a, b)\in\mathcal A$.
 Then there exists $\theta\in(0, 1)$ such that  for any
$\varphi,\psi \in \mathcal C(\mathbb R^2)$
 \be \label{bound} C_n(\varphi, \psi;
\mu)\le 2(\|\varphi\|_\infty+\|\psi\|_\infty)\mathcal R_{\varphi,
\psi}(\theta^n)+C{\theta}^n
 \ee
where constant $C$ depends on $\varphi,$ $\psi$ and $f.$
\end{theorem}

The following corollary is a direct application of the above theorem.

\begin{corollary}\label{hen}
Let $f=T_{a,b}$ and $(a, b)\in\mathcal A.$
\begin{itemize}
\item[(i)]If $\mathcal R_{\varphi, \psi}(\varepsilon)\lesssim
e^{\alpha\log\varepsilon}$ i.e. if the observables $\varphi, \psi$
are H\"older continuous with exponent $\alpha$ then $C_n(\varphi,
\psi; \mu)\lesssim e^{-{\alpha' n}},$$\alpha'=\alpha|\log\theta|.$
\item[(ii)] If\,\, $\mathcal R_{\varphi, \psi}(\varepsilon)\lesssim e^{-|\log
\varepsilon|^{\alpha}},$ $\alpha\in(0, 1)$  then $C_n(\varphi, \psi;
\mu)\lesssim e^{-n^\alpha}.$
\item[(iii)] If\,\, $\mathcal R_{\varphi, \psi}(\varepsilon)\lesssim |\log
\varepsilon|^{-\alpha},$ for $\alpha>0$  then $C_n(\varphi, \psi;
\mu)\lesssim {n^{-\alpha}}.$
\end{itemize}
\end{corollary}
Notice that the upper in (i) is the same as in \cite{BY2}, which is
natural to expect. The estimates in (ii)-(iii) are new.
\subsection{Solenoid with intermittency}
The second class of maps we consider is a solenoid map with an
indifferent fixed point  \cite{AlvPin}, which is defined  as
follows. Let $M=\mathbb S^1\times D^2,$ where $D^2$ is a unit disk
in $\mathbb R^2.$ For $(x, y, z)\in M$ let
$$g(x, y, z)=\left(f(x), \frac{1}{10}y+\frac{1}{2}\cos x,
\frac{1}{10}z+\frac 1 2\sin x\right),$$
where $f:\mathbb S^1\to
\mathbb S^1$ is a map of degree $d\ge 2$ with the properties: $f$ is
$C^2$ on $\mathbb S^{1}\setminus \{0\};$ $f$ is $C^1$ on $\mathbb
S^1$ and $f'>1$ on $\mathbb S^1\setminus \{0\};$ $f(0)=0,$ $f'(0)=1$
and there exists $\gamma>0$ such that $-xf''\approx |x|^\gamma$ for
all $x\neq 0.$

\begin{theorem}\label{main1} Let $g$ be the map described above.
Assume that $\gamma<1.$ Then for any $\varphi,\psi \in \mathcal
C(M)$
 \be \label{bound1} C_n(\varphi, \psi;
\mu)\le 2(\|\varphi\|_\infty+\|\psi\|_\infty)\max\{\mathcal
R_\varphi(n^{-1/\gamma}), \mathcal
R_\psi(n^{-1/\gamma})\}+Cn^{1-1/\gamma}
 \ee
where constant $C$ depends on $\varphi$ and $\psi.$
\end{theorem}

Direct application of the theorem for specific classes of
observables gives.

\begin{corollary}\label{sol}
Let $g$ be a map as in the theorem above. Then
\begin{itemize}
\item[(i)]If $\mathcal R_{\varphi, \psi}(\varepsilon)\lesssim
e^{\alpha\log\varepsilon}$ i.e. the observables $\varphi, \psi$ are
H\"older continuous with exponent $\alpha$ then $C_n(\varphi, \psi;
\mu)\lesssim e^{-\min\{\alpha/\gamma, 1/\gamma-1\}\log n},$

\item[(ii)] If\,\, $\mathcal R_{\varphi, \psi}(\varepsilon)\lesssim e^{-|\log
\varepsilon|^{\alpha}},$ $\alpha\in(0, 1)$  then $C_n(\varphi, \psi;
\mu)\lesssim e^{-(\log n)^\alpha}.$

\item[(iii)] If\,\, $\mathcal R_{\varphi, \psi}(\varepsilon)\lesssim |\log
\varepsilon|^{-\alpha},$ for $\alpha>0$  then $C_n(\varphi, \psi;
\mu)\lesssim |\log n|^{-\alpha}.$
\end{itemize}
\end{corollary}

The estimate in the first item coincides with the one given in
\cite{AlvPin} the remaining two are new.
%Comparing the corollaries \ref{hen} and \ref{sol} we can see that
%when the hyperbolic qualities of the map is lower then the
%regularity of the observable affects significantly.

\subsection{Young Towers} We now define GMY-structure
\cite{AlvPin}  which generalizes the definition of \cite{Y1} (see
Remarks 2.3-2.5 in \cite{AlvPin}). Let $f:M\to M$ be a
diffeomorphism of Riemannian manifold $M.$ If $\gamma\subset M$ is a
submanifold, then $m_\gamma$ denotes the restriction of the
Riemannian volume to $\gamma.$ Assume that $f$ satisfies the
following conditions.

\begin{itemize}
\item[(A1)] There exists $\Lambda\subset M$ with hyperbolic product
structure, i.e. there are continuous families of stable and unstable
manifolds $\Gamma^s=\{\gamma^s\}$ and $\Gamma^u=\{\gamma^u\}$ such
that $\Lambda=(\cup\gamma^s)\cap(\cup\gamma^u);$
dim$\gamma^s+\text{dim}\gamma^u=\text{dim}M;$ each $\gamma^s$ meets
each $\gamma^u$ at a unique point; stable and unstable manifolds are
transversal with angles bounded away from $0;$
$m_{\gamma}(\gamma\cap \Lambda)>0$ for any $\gamma\in\Gamma^u.$
\end{itemize}
Let $\Gamma^s$ and $\Gamma^u$ be the defining families of $\Lambda.$
A subset $\Lambda_0\subset\Lambda$ is called $s$-subset if
$\Lambda_0$ also has a hyperbolic structure and its defining
families can be chosen as $\Gamma^u$ and
$\Gamma^s_0\subset\Gamma^s.$ Similarly, we define $u$-subsets. For
$x\in\Lambda$ let $\gamma^\sigma(x)$ denote the element of
$\Gamma^\sigma$ containing $x,$ where $\sigma\in \{u, s\}.$

\begin{itemize}
\item[(A2)] There are pairwise disjoint $s$-subsets $\Lambda_1, \Lambda_2,
..., \subset\Lambda$ such that \\$m_{\gamma^u}((\Lambda\setminus\cup
\Lambda_i)\cap\gamma^u)=0$ on each $\gamma^u,$ and for each
$\Lambda_i,$ $i\in\mathbb{N}$ there is $R_i$ such that
$f^{R_i}(\Lambda_i)$ is $u$-subset;
$f^{R_i}(\gamma^s(x))\subset\gamma^s(f^{R_i}(x))$ and
$f^{R_i}(\gamma^u(x))\supset\gamma^u(f^{R_i}(x))$ for any
$x\in\Lambda_i.$ The separation time $s(x, y)$ is the smallest $k$
where $(f^R)^k(x)$ and $(f^R)^k(y)$ lie in different partition
elements.

\item[(A3)] There exist constants $C> 0$ and $\beta\in(0, 1)$ such that
 \\$\text{dist}(f^n(x), f^n(y))\le C\beta^n,$ for all
$y\in\gamma^s(x)$ and $n\ge 0.$
\end{itemize}

Let $f^u$ denote the restriction of $f$ onto unstable disks, and
$Df^u$ its differential.

\begin{itemize}
\item[(A4)] Regularity of the stable foliation:
given $\gamma, \gamma'\in \Gamma^u$ define $\Theta:\gamma'\cap
\Lambda\to \gamma\cap \Lambda$ by $\Theta(x)=\gamma^s(x)\cap\gamma.$
Then
\begin{itemize}
\item[(a)] $\Theta$ is absolutely continuous and
$$u(x):=\frac{d(\Theta_\ast m_{\gamma'})}{dm_\gamma}(x)=\prod_{i=0}^{\infty}
\frac{\det Df^u(f^i(x))}{\det Df^u(f^i(\Theta^{-1}(x)))};$$
\item[(b)] We assume that there exists $C>0$ and $\beta\in(0, 1)$
such that
$$\log\frac{u(x)}{u(y)}\le C\beta^{s(x, y)}\quad \text{for}\quad x, y\in\gamma'\cap \Lambda.$$
\end{itemize}

\item[(A5)] Bounded distortion: for $\gamma\in\Gamma^u$ and $x, y\in \Lambda\cap\gamma$
$$\log\frac{\det D(f^R)^u(x)}{\det D(f^R)^u(y)}\le C\beta^{s(f^R(x), f^R(y))}.$$

\item[(A6)] Integrability: $\int R\,dm_0<\infty.$
\item[(A7)] Aperiodicity: gcd$\{R_i\}=1.$
\end{itemize}

The geometric structure described in (A1) and (A2) allows us to
define the corresponding \emph{Young Tower}. More precisely, we let
\be\label{tower1}
 \mathcal{T}=\{(z,\ell)\in\Lambda\times\mathbb{Z}^+_0|\,\, R(z)>\ell\},
 \ee
where $\mathbb{Z}_0^+$ denotes the set of all nonnegative integers.
For $\ell\in\Z_0^+$ the subset $\mathcal
T_\ell=\{(\cdot,\ell)\in\mathcal T\}$ of $\mathcal T$ is called its
{\it$\ell$th level} (identify $\Lambda$ with $\mathcal T_0$ and
$\Lambda_i$ with $\mathcal T_{0, i}$). The sets $\mathcal
T_{\ell,i}:=\{(z,\ell)\in\mathcal T_\ell|\,\,(z, 0)\in\mathcal
T_{0,i}\}$  give a partition $\mathcal P$ of $\mathcal T.$ We can
define a map \( F: \mathcal T \to \mathcal T \) letting
\begin{equation}\label{towerm}
F(z,\ell)=\Big\{\begin{array}{cc}
(z,\ell+1) &\text{if} \,\, \ell+1<R(z),\\
(f^{R(z)}(z), 0)&  \text{if}\,\, \ell+1=R(z).
\end{array}
\end{equation}
A measure  \(  m  \) on \(  \mathcal T\) is defined as
follows. Let $\mathcal{A}$ be a $\sigma$-algebra on $\Lambda
\subseteq M,$ and let $m_\Lambda$ denote a finite measure on  $\Lambda.$  For  any $\ell\geq 0$ and $A\subset\mathcal
T_\ell$ such that
 $F^{-\ell}(A)\in\mathcal{A}$ define
$m(A)=m_\Lambda(F^{-\ell}(A)).$ There is  tower projection
$\pi:\mathcal T\to M,$ which is a semi-conjugacy $f\circ\pi=\pi\circ
F,$ defined as follows: for $x\in\mathcal T$ with $x_0\in\Lambda$
and $F^\ell(x_0)=x$ we let $\pi(x, \ell)=f^\ell(x_0).$

Notice that it is not strictly necessary for \(  f  \) to be a
diffeomorphism. Such a structure can be defined for example also in
the presence of some discontinuities as long as the stable and
unstable manifolds exist and satisfy the required properties, this
has been done for example for a class of Billiards in \cite{Chern}.

It is known from \cite{Y1} that under the above assumptions $f$
admits an SRB measure $\mu$ and here we study further properties of
the measure $\mu.$ Now, for $n\ge 0$ introduce a sequence of
partitions as follows:
$$\mathcal P_0=\mathcal P\quad \text{and} \quad \mathcal
P_n=\bigvee_{i=0}^{n-1}F^{-i}\mathcal P.$$

Let
$$\delta_n=\sup\{\text{diam}\pi(F^n(P)):P\in \mathcal P_{2n}\}.$$

Now we state the main technical theorem which is of independent
interest.

\begin{mtthm} Let $f:M\to M$ be a
diffeomorphism of Riemannian manifold $M,$ which admits a Young
tower. Then for any non-zero observables $\varphi,\psi\in\mathcal
C(M)$ we have
\be \label{corrl}
C_n(\varphi, \psi; \mu)\le
2(\|\varphi\|_\infty+\|\psi\|_\infty)\max\{\mathcal
R_\varphi(\delta_n), \mathcal R_\psi(\delta_n)\}+u_n\ee
 where $u_n$ is
a  sequence  of positive numbers defined as follows:
\begin{itemize}
\item[(i)] If $m\{R>n\}\le C\theta^n$ for some  $C>0$ and
$\theta\in(0, 1),$  then there exist $\theta'\in (0, 1)$ and $C'>0$
such that $u_n\le C'{\theta'}^n.$
\item[(ii)] If $m\{R>n\}\le C e^{-cn^\eta},$ for some $C,c>0$ and
$\eta\in (0, 1),$ then there are $C',c'>0$ such that $u_n\le
C'e^{-c'n^\eta}.$
\item[(iii)] If $m\{R>n\}\le Cn^{-\alpha}$ for some
$C>0$ and $\alpha>1$  then  there exists $C'>0$ such that $u_n\le
C'n^{1-\alpha}.$
\end{itemize}
\end{mtthm}

\section{Proof of  main technical theorem}
In this section we reduce the system to a non-invertible system, as
in \cite{Y1}. We start by defining a special measure on $\Lambda.$

\subsection{The natural measures on the  unstable manifolds}
We fix  $\hat\gamma\in\Gamma^u.$ For any $\gamma\in\Gamma^u$ and
$x\in\gamma\cap\Lambda$ let $\hat x$ be the point
$\gamma^s(x)\cap\gamma.$ Define for $x\in\gamma\cap\Lambda$
$$
\hat u(x)=\prod_{i=0}^{\infty}\frac{\det Df^u(f^i(x))}{\det
Df^u(f^i(\hat x))}.
$$
By item (b) of assumption (A4) $\hat u$ satisfies the bounded
distortion property. For each $\gamma\in\Gamma^u$ define the measure
$m_\gamma$ as
$$
\frac{dm_\gamma}{d\text{Leb}_\gamma}=\hat
u\mathbf{1}_{\gamma\cap\Lambda},
$$
where $\mathbf{1}_{\gamma\cap\Lambda}$ is the characteristic
function of $\gamma\cap\Lambda.$ The proof of the following lemma is
fairly standard and can be found in \cite{AlvPin}.

\begin{lemma}\label{mgamma} \begin{itemize}
\item[(i)] Let $\Theta$ be the map defined in $(A4).$
Then $\Theta_\ast m_\gamma=m_{\gamma'}$ for any
$\gamma,\gamma'\in\Gamma^u.$

\item[(ii)] Let $\gamma, \gamma'\in\Gamma^u$  be such that
$f^R(\gamma\cap\Lambda)\subset\gamma',$ and let $Jf^R(x)$ denote the
Jacobian of $f^R$ with respect to the measures $m_\gamma$ and
$m_{\gamma'}.$ Then $Jf^R(x)$ is constant on the stable manifolds
and there is $C>0$ such that for every $x,y \in\gamma\cap\Lambda$
$$\left|\frac{Jf^R(x)}{Jf^R(y)}-1\right|\le C\beta^{s(f^R(x),
f^R(y))}.$$
\end{itemize}
\end{lemma}

\subsection{Quotient tower} Let $\bar\Lambda=\Lambda/\sim,$ where
$ x \sim y$ if and only if $y\in\gamma^s(x).$ This equivalence
relation gives rise to a \emph{quotient tower}
$$
\Delta=\mathcal T/\sim
$$
 with $\Delta_\ell=\mathcal T_\ell/\sim$ and
its partition into $\Delta_{0, i}=\mathcal T_{0,i}/\sim$ which we
denote by $\bar{\mathcal P}.$ There is a natural projection $\bar
\Pi:\mathcal T\to \Delta.$ Since $f^R$ preserves stable leaves and
$R$ is constant on them the return time $\bar R$ and separation time
$\bar s$ are well defined by $R$ and $s.$ Moreover, we can define a
tower map $\bar F:\Delta\to \Delta.$ Let $m$ be a measure whose
restriction onto unstable manifolds is $m_\gamma.$  Lemma
\ref{mgamma} implies that there is a measure $\bar m$ on $\Delta$
whose restriction to each $\gamma\in \Gamma^u$ is $m_\gamma.$ We let
$J\bar F$ denote the Jacobian of $\bar F$ with respect to $\bar m.$

% \begin{lemma} There exists $C>0$ such that for all $k\ge 1$ and all $x, y\in \bar \Delta$ belonging to the same
% element of $\mathcal Q_{k-1}$
% $$\left|\frac{J\bar F^k(x)}{J\bar F^k(y)}-1\right|\le C\beta^{\bar s(\bar F^k(x),
% \bar F^k(y))}.$$
% \end{lemma}
Next we introduce the space of H\"older continuous functions on
$\Delta$  as
$$
\mathcal F_\beta=\{\varphi:\Delta\to \mathbb{R}: \exists C_\varphi
\,\, \text{such that}\,\, |\varphi(x)-\varphi(y)|\le
C_\varphi\beta^{\bar s(x, y)} \,\, \forall x,y\in\Delta \}.
$$

\begin{align*}
\mathcal F^+_\beta=\big\{\varphi\in\mathcal F: \exists C_\varphi\,\,
\text{such that on each} \,\,\Delta_{\ell, i}, \,\,
\text{either}\,\, \varphi\equiv 0,\,\, \text{or}\\
\varphi>0, \left|\frac{\varphi(x)}{\varphi(y)}-1\right|\le C_\varphi
 \beta^{\bar s(x, y)} \,\,\forall x,y\in \Delta_{\ell,i} \big\}.
\end{align*}

The mixing properties of  $\bar F$ was studied firstly in \cite{Y2}.
Several papers appeared improving or extending the results given in
\cite{Y2}, for example \cite{Gou, Hol, Ly}. Here we combine the
results from \cite{AlvPin} and \cite{Y2}.

\begin{theorem} \cite{AlvPin,Y2}\label{speed}
\begin{itemize}
\item[(i)] $\bar F$ admits unique mixing acip $\bar \nu;$
 $d\bar\nu/d\bar m\in\mathcal F^+$ and $d\bar\nu/d\bar m>c>0.$
 \item[(ii)] Let $\lambda$
be a probability  measure with $\varphi=d\lambda/d\bar m\in \mathcal
F^+_\beta.$
\begin{enumerate}
\item If $\bar m\{\bar R>n\}\le C\theta^n$ for some $C>0$ and
$\theta\in(0, 1)$ then there exists $C'>0$ and $\theta'\in(0, 1)$
such that $|\bar F_\ast^n-\nu|\le C'\theta'^n.$
\item If $\bar m\{\bar R>n\}\le Ce^{-cn^\eta}$ for some $C, c>0$ and
$\eta\in(0, 1]$ then there exists $C',c'>0$ such that $|\bar
F_\ast^n-\nu|\le C'e^{-c'n^\eta}.$ Moreover $c'$ does not depend on
$\varphi,$ $C'$ depends only on $C_\varphi.$
\item If $\bar
m\{\bar R>n\}\le Cn^{-\alpha}$ for some $C>0$ and $\alpha>1$ then
there exists $C'>0$ such that$ |\bar F_\ast^n-\nu|\le
C'n^{1-\alpha}.$
\end{enumerate}
\end{itemize}
\end{theorem}

\subsection{Approximation of correlations}
We establish the relation between the original problem and problem
of estimating the decay rates of correlations on the quotient tower,
and then we apply theorem \ref{speed}.

Let $\pi:\mathcal T\to M,$ $\bar\pi: \mathcal T\to \Delta$ be the
tower projections, then we have  $\bar\nu=\bar\pi_\ast\nu$ and
$\mu=\pi_\ast\nu.$ Given $\varphi, \psi\in \mathcal C^0(M)$ define
$\tilde\varphi=\varphi\circ\pi$ and $\tilde\psi=\psi\circ\pi.$ By
definition
$$\int(\varphi\circ f^n)\psi d\mu-\int\varphi d\mu\int\psi d\mu=
\int(\tilde\varphi\circ F^n)\tilde\psi d\nu-\int\tilde\varphi
d\nu\int\tilde\psi d\nu,$$ which shows it is sufficient to obtain
estimates for the lifted observables. This will be done by
approximating the lifted observables with piecewise constant
observables on tower. For $k\le n/4$ define $\bar\varphi_k$ as
follows
$$
\bar\varphi_k|_P=\inf\{\tilde\varphi\circ F^k(x)| \,\, x\in P\},
\,\,\text{where} \,\, P\in\mathcal P_{2k}.
$$
Define $\bar\psi_k$ in a similar way, and note that $\bar\varphi_k$
and $\bar\psi_k$ are constant on the stable leaves. Hence, we can
consider them as a function defined on quotient tower $\Delta.$ The
main result of this section is the following
\begin{proposition}\label{aprox}
$$|\mathcal C_n(\tilde\varphi, \tilde\psi; \nu)-
\mathcal C_n(\bar\varphi_k, \bar\psi_k;\bar\nu)|\le
2(\|\varphi\|_\infty+\|\psi\|_\infty)\max\{\mathcal
R_\varphi(\delta_k), \mathcal R_\psi(\delta_k)\}.$$
\end{proposition}

\begin{proof}
The proof consists of several steps and follows the argument in
\cite{AlvPin}. First we claim \be\label{diff1} |C_n(\tilde\varphi,
\tilde\psi; \nu)- C_{n-k}(\bar\varphi_k, \tilde\psi; \nu)|\le 2
\|\psi\|_\infty\mathcal R_\varphi(\delta_k). \ee

Indeed, using the fact $C_n(\tilde\varphi, \tilde\psi;
\nu)=C_{n-k}(\tilde\varphi\circ F^k, \tilde\psi; \nu)$ the left hand
side of \ref{diff1} can be written as
\begin{align*}\left|\int(\tilde\varphi\circ
F^k-\bar\varphi_k)\tilde\psi d\nu+\int(\tilde\varphi\circ
F^k-\bar\varphi_k)d\nu\int\tilde\psi d\nu\right|\\ \le
2\|\tilde\psi\|_\infty\int|\tilde\varphi\circ
F^k-\bar\varphi_k|d\nu.
\end{align*}
By definition of $\bar\varphi_k$ for $x\in P$ we have
$$|\tilde\varphi\circ F^k(x)-\bar\varphi_k|\le
\sup_{x, y\in P}|\tilde\varphi(F^k(x))-\tilde\varphi(F^k(y))|\le
\mathcal R_\varphi(\delta_k),$$ which implies desired conclusion.
Now, let $\bar\psi_k\nu$ be the measure whose density with respect
to $\nu$ is $\bar\psi_k$ and let $\tilde\psi_k=
dF_\ast^k(\bar\psi_k\nu)/d\nu.$ Then \be\label{diff2}
|C_{n-k}(\bar\varphi_k, \tilde\psi; \nu)- C_{n-k}(\bar\varphi_k,
\tilde\psi_k; \nu)|\le 2 \|\varphi\|_\infty\mathcal
R_\psi(\delta_k). \ee After substituting and simplifying the
expression we obtain
$$|C_{n-k}(\bar\varphi_k, \tilde\psi; \nu)- C_{n-k}(\bar\varphi_k,
\tilde\psi_k; \nu)|\le
2\|\varphi\|_\infty\left|\int(\tilde\psi-\tilde\psi_k)d\nu\right|.$$
First observe that $F^k_\ast((\tilde\psi\circ
F^k)\nu)=\tilde\psi\nu.$ Letting $|\cdot|$ denote the variational
norm for measures we have
\begin{align*}
\left|\int(\tilde\psi-\tilde\psi_k)d\nu\right|
=|F^k_\ast((\tilde\psi\circ F^k)\nu)-F^k_\ast(\bar\psi_k\nu)|
\\
\le |(\tilde\psi\circ F^k-\bar\psi_k)\nu|=\int|\psi\circ
F^k-\bar\psi_k|d\nu.
\end{align*}
As in the proof of \eqref{diff1} we have  $|\psi\circ
F^k-\bar\psi_k|\le\mathcal R_\psi(\delta_k)$ which implies relation
\ref{diff2}. Combining the inequalities \eqref{diff1}, \eqref{diff2}
and  the equality $C_{n-k}(\bar\varphi_k, \tilde\psi_k;
\nu)=C_n(\bar\varphi_k, \bar\psi_k; \bar\nu)$ from \cite{AlvPin}
finishes the proof.
\end{proof}

It remains to prove decay of correlations for the observables
$\bar\varphi_k$ and $\bar\psi_k$ on the quotient tower. We start
with the usual transformations that simplify the correlation
function. Without lost of generality assume that $\bar\psi_k$ is not
identically zero. Let
$b_k=\left(\int(\bar\psi_k+2\|\bar\psi_k\|_{\infty})d\bar\nu\right)^{-1}$
and $\hat\psi_k=b_k(\bar\psi_k+2\|\bar\psi_k\|_\infty),$ then we
have $\int\hat\psi_kd\bar\nu=1,$  $\|\bar\psi_k\|\le b_k^{-1}\le
3\|\bar\psi_k\|$ and $1\le \hat\psi_k\le 3.$ Moreover, $\hat\psi_k$
is constant on the elements of $\mathcal P_{2k}.$ Thus
\be\begin{split} C_n(\bar\varphi_k, \bar\psi_k; \bar\nu)=
\left|\int(\bar\varphi\circ \bar
F^n)\bar\psi_kd\bar\nu-\int\bar\varphi_kd\bar\nu
\int\bar\psi_kd\bar\nu\right|= \\
\frac{1}{b_k}\left|\int(\varphi\circ \bar
F^n)\hat\psi_kd\bar\nu-\int\bar\varphi_kd\bar\nu \right|\le\\
\frac{1}{b_k}\|\bar\varphi_k\|_{\infty}\int \left| \frac{d\bar
F_\ast^n(\hat\psi_k\bar\nu)}{d\bar m}-\frac{d\bar\nu}{d\bar
m}\right|d\bar m.
\end{split} \ee
Now, letting $\hat\lambda_k=\bar F^{2k}_\ast(\hat\psi_k\bar\nu)$ we
conclude \be\label{ineq}
 C_n(\bar\varphi_k, \bar\psi_k; \bar\nu)\le
\frac{1}{b_k}\|\varphi\|_\infty\left|F^{n-2k}_\ast\hat\lambda_k-\bar\nu\right|.
\ee Note that  the density of $\hat\lambda_k$ belongs to the class
$\mathcal F^+$ (see Lemma 4.1, \cite{AlvPin}). Hence we can apply
 theorem \ref{speed}  to
 $\left|F^{n-2k}_\ast\hat\lambda_k-\bar\nu\right|$   and
 obtain desired estimates for  $C_n(\bar\varphi_k, \bar\psi_k;
\bar\nu).$

\section{Proofs of theorems \ref{main} and \ref{main1}}
We start this section with the following auxiliary construction.
Consider the sequence of stopping times for the points in $\Lambda$
defined as follows:
\begin{equation}
S_0=0, \,\, S_1=R\,\,\text{and}\,\, S_{i+1}=S_i+R\circ f^{S_i},\,\,
\text{for} \,\, i\ge 1.
\end{equation}
Let $\mathcal Q_0$ be the partition of $\Lambda$ into
$\Lambda_{i}$'s. Define the sequence of partitions  $\mathcal Q_k$
as:
 $x, y\in\Lambda$ belong to the same element of $\mathcal Q_k$ if
the following conditions hold.
\begin{itemize}
\item[(i)] $f^R(x)$ and $f^R(y)$ have the same stopping times
 up to time $k-1.$
 \item[(ii)] $f^{S_i}(f^R(x))$ and $f^{S_i}(f^R(y))$ belong to the same
 element of $\mathcal Q_0$ for each $0\le i\le k-1.$
 \item[(iii)] $f^{S_k(Q)}$ is $u$-subset.
\end{itemize}
For $Q\in\mathcal Q_0$ let $R(Q)$ denote its return time. Let $k\ge
1$ be arbitrary integer and define  a sequence
$$\bar\delta_k=\sup_{Q\in\mathcal Q_0}\bar\delta_k(Q),$$ where $\bar\delta_k(Q)$ is
defined as follows:
\begin{enumerate}
\item For $k>R(Q)-1$,  let
$$\bar\delta_k:=\sup_{0\le \ell\le R(Q)-1}
\{\text{diam}(f^\ell(A\cap\gamma)):\gamma\in \Gamma^u, A\in\mathcal
Q_{k-R(Q)+1+\ell}, A\subset Q\}.$$

\item For $k\le R(Q)-1,$  let
$$\bar\delta_k^0(Q):=\sup_{0\le \ell <R(Q)-k}
\{\text{diam}(f^\ell(Q\cap\gamma)):\gamma\in\Gamma^u\},$$
$$\bar\delta_k^+:=\sup_{R(Q)-k \le \ell\le R(Q)-1}
\{\text{diam}(f^\ell(A\cap\gamma)):\gamma\in \Gamma^u, A\in\mathcal
Q_{k-R(Q)+1+\ell}, A\subset Q\}$$
\end{enumerate}
and  define $$\bar\delta_k(Q)=\sup\{\bar\delta_k^0(Q),
\bar\delta_k^+(Q)\}.$$ From Lemma 3.2 in  \cite{AlvPin} we have

\be\label{diameter}
 \text{diam}(\pi(F^k(P))\le C\max\{\beta^k,
\bar\delta_k\} \ee for any $P\in\mathcal P_{2k}$, $k\ge 0,$ and some
$C>0.$ This is the main estimate we use to prove theorems \ref{main}
and \ref{main1}.

In \cite{AlvPin} it was proven that $\bar\delta_k\lesssim
k^{-1/\gamma}$  and $m\{R>k\}\le k^{-1/\gamma}$ for the Solenoid map
with intermittent fixed point. Substituting this into
\eqref{diameter}, we can apply item (ii) of main technical theorem
and conclude the proof of theorem \ref{main1}.

In \cite{BY2} it was shown that for any $(a, b)\in\mathcal A$ the
corresponding H\'enon map admits a Young tower, for which, the tail
of the return time decays exponentially. Therefore to complete the
proof of theorem \ref{hen}  it is sufficient to show that
$\bar\delta_k$ decays exponentially. We will show this in the
following lemma and complete the proof in this case also.

In  \cite{Y1} (see \cite{BY2} for the details of the construction)
it was shown that H\'enon maps satisfy backward contraction on the
unstable leaves, that is there exits $C>0$ such that for all $x,
y\in \Lambda_i$ with $y\in\gamma^u(x)$ and $0\le n\le R_i$
\be\label{contr} \text{dist}(f^n(x), f^n(y))\le C\beta^{R_i-n}. \ee

\begin{lemma}
$\exists C>0$ and  $\beta'\in(0, 1)$ such that $\bar\delta_k\le
C\beta'^k.$
\end{lemma}

\begin{proof}

A) We start with the case $k\le R(P)-1$ and $0\le \ell< R(P)-k.$ By
\eqref{contr} for any $x\in P$
$$
\text{diam}(f^\ell(P\cap\gamma^u(x)))\le C\beta^{R(P)-\ell}\le
C\beta^k.$$

This implies $\bar\delta_k^0\lesssim \beta^k.$

B) Now, consider the case $k\le R(P)-1,$ and $R(P)-k< \ell \le
R(P)-1.$ Notice that for any $Q\subset P$, $Q\in\mathcal
P_{k-R(P)+\ell+1}$ the stopping times $S_1, ..., S_{\ell'},$
$\ell'=k-R(P)+\ell,$ are constant on $Q$ and $f^{S_i}(Q)\subset
P_i,$ for some $P_i\in\mathcal Q_0,$ $i=1, ..., \ell'.$ Let $r_1,
..., r_{\ell'-1}$ be the return times of these  elements. By
\eqref{contr} we have
$$
\text{diam}(Q\cap\gamma^u) \le C\beta^{R(P)+r_1+ ... +r_{\ell'-1}}
$$

Since $R(P)+r_1+...+r_{\ell'-1}\ge l+k$ using again the
inequality \eqref{contr}
$$\text{diam}(f^{\ell}(Q\cap\gamma^u))\le C\beta^{R(P)+r_1+ ... +r_{\ell'-1}}\le C\beta^k.$$
This implies $\bar\delta_k^+\lesssim \beta^k,$ which  finishes the proof when $k\le R(P).$ The
case $k>R(P)$ is treated as B).
\end{proof}

\subsection*{Acknowledgements}Many thanks  to my supervisor S. Luzzatto for his
encouragements and guidance.

\end{document}